\documentclass[12pt, a4paper, leqno]{amsart}
\usepackage{amsmath}
\usepackage{amsfonts}
\usepackage{amssymb}
\usepackage{txfonts}
\usepackage{comment}
\usepackage{ae}
\usepackage{graphicx}
\usepackage[colorlinks]{hyperref}
 
\let\oldmarginpar\marginpar
\renewcommand\marginpar[1]{\-\oldmarginpar[\raggedleft\footnotesize #1]%
{\raggedright\footnotesize #1}}

\newcommand{\R}{{\varmathbb{R}}}
\newcommand{\N}{{\varmathbb{N}}}
\newcommand{\Z}{\varmathbb{Z}}
\newcommand{\Rn}{\varmathbb{R}^n}

\newcommand{\W}{\mathcal{W}}

\newcommand{\C}{\mathcal{C}}

\def\diam{\qopname\relax o{diam}}

\def\dist{\qopname\relax o{dist}}
\def\min{\qopname\relax o{min}}
\def\diam{\qopname\relax o{diam}}
\def\inte{\qopname\relax o{int}}

\usepackage{amsthm}

\swapnumbers
\theoremstyle{plain}
\newtheorem{thm}[equation]{Theorem}
\newtheorem{lem}[equation]{Lemma}
\newtheorem{prop}[equation]{Proposition}
\newtheorem{cor}[equation]{Corollary}

\theoremstyle{definition}
\newtheorem{defn}[equation]{Definition}

\theoremstyle{remark}
\newtheorem{rem}[equation]{Remark}

\numberwithin{equation}{section}

\title{Poincar\'e inequalities in quasihyperbolic boundary condition
  domains} 

\author{Ritva Hurri-Syrj\"anen}\address[R. H.-S.]{University of
  Helsinki, Department of Mathematics and Statistics, P.O. Box 68
  (Gustaf H\"allstr\"omin katu 2b), FI-00014 University of Helsinki,
  Finland} \email{ritva.hurri-syrjanen@helsinki.fi}

\author{Niko Marola}\address[N. M.]{University of Helsinki, Department
  of Mathematics and Statistics, P.O. Box 68 (Gustaf H\"allstr\"omin
  katu 2b), FI-00014 University of Helsinki, Finland}
\email{niko.marola@helsinki.fi}

\author{Antti V. V\"ah\"akangas} \address[A. V. V.]{University of
  Helsinki, Department of Mathematics and Statistics, P.O. Box 68
  (Gustaf H\"allstr\"omin katu 2b), FI-00014 University of Helsinki,
  Finland} \email{antti.vahakangas@helsinki.fi} \thanks{A. V. V.  was
  supported by the Academy of Finland, grants 75166001 and 1134757,
  and by the Finnish Academy of Science and Letters, Vilho, Yrj\"o and
  Kalle V\"ais\"al\"a Foundation}

\begin{document}


\begin{abstract}
  We study the validity of $(q,p)$-Poincar\'e inequalities, $q<p$, on
  domains in $\R^n$ which satisfy a quasihyperbolic boundary
  condition, i.e. domains whose quasihyperbolic metric satisfies a
  logarithmic growth condition. Maz'ya has given an implicit
  characterization for domains supporting $(q,p)$-Poincar\'e
  inequalities, $q<p$; in the present paper, we show that the
  quasihyperbolic boundary condition domains are such domains whenever
  $p>p_0$, where $p_0$ is an explicit constant depending on $q$, on
  the logarithmic growth condition, and on the boundary of the domain.
\end{abstract}

\maketitle

\markboth{\textsc{R. Hurri-Syrj\"anen, N. Marola, and A. V.~V\"ah\"akangas}}
{\textsc{Poincar\'e inequalities and quasihyperbolic boundary condition}}

\section{Introduction}

A bounded domain $G$ in $\Rn$, $n\geq 2$, is said to support a
$(q,p)$-Poincar\'e inequality if there exists a finite constant $c$
such that the inequality
\begin{equation}\label{poincare}
\left (\int_G |u(x)-u_G|^q\,dx\right)^{1/q} 
	\le c  \left ( \int_G |\nabla u(x)|^p\,dx
\right)^{1/p}
\end{equation}
holds for all functions $u$ in the Sobolev space $W^{1,p}(G)$; here
$1\le p, q <\infty$ and $u_G$ is the integral average of $u$ over
$G$. If $G$ is a John domain (see Definition~\ref{sjohn}), then it is
well known that \eqref{poincare} is valid for all $(q,p)$ where $1\le
p\le q\le np/(n-p)$ \cite[Theorem 5.1]{B}. Property \eqref{eq:John} of
John domains implies that a Poincar\'e inequality supported by balls
is valid also in John domains. In this paper we consider a larger
class of domains which do not inherit the inequalities which balls
support; we study bounded domains satisfying the quasihyperbolic
boundary condition, see Definition~\ref{qhbc}.

A proper subclass of quasihyperbolic boundary condition domains is
formed by John domains, but domains in the former class allow narrow
gaps which can destroy the John condition \eqref{eq:John},
\cite[Example 2.26]{GM}.  This kind of effect implies that a
$(p,p)$-Poincar\'e inequality fails to hold for small values of $p$,
whereas the domain does support the $(p,p)$-Poincare inequality for
large enough $p$.

 A $(q,p)$-Poincar\'e inequality is valid in a
  $\beta$-quasihyperbolic boundary condition domain, if $n-n\beta
  <q=p<\infty$, see \cite[Theorem 1.4]{KOT}, and also \cite[Remark
  7.11]{H}; and if $n-nb <p\le q < b np/(n-p) $, whenever $p<n$ and
  $b=2\beta/(1+\beta)$ \cite[Theorem 1]{JK}; see also \cite[Theorem
  1.5]{KOT}, \cite[Theorem 1.4]{H-S}. It is shown in \cite{JK} that if
  $1\le p<n-nb$, then there exist $\beta$-quasihyperbolic boundary
  condition domains which do not support the $(p,p)$-Poincar\'e
  inequality. We remark that $\beta$-quasihyperbolic boundary
condition domains support $(1,p)$-Poincar\'e inequality for all $p
>n-nb$ by H\"older's inequality while John domains support
$(1,p)$-Poincar\'e inequality for all $1\leq p<\infty$. The question
one may ask is, what can be said about the validity of
$(q,p)$-Poincar\'e inequalities in the case $1\le q < \min \{n-n b,
p\}$.

Poincar\'e inequalities, \eqref{poincare}, in the case $1\leq q <p$
have been considered on general domains, e.g., in \cite[Section
6.4]{M}, see also \cite{HH-S} and the $(1,p)$-case in
\cite{HH-SV}. Maz'ya~\cite{M}, Theorem 6.4.3/2 on p. 344, gives a
characterization for domains which support \eqref{poincare} when
$q<p$. In addition, this class of domains characterizes certain
compact embeddings, see Theorem 6.8.2/2 on p. 376 \cite{M} for more
details. Maz'ya presents also applications to the Neumann problems for
strongly elliptic operators in domains which characterize
\eqref{poincare} with $p=2$ and $1<q\le 2$, cf. Section 6.10.1. We
shall discuss applications in Section~\ref{sect:appls}.

In the present paper, we answer the question about $(q,p)$-Poincar\'e
inequality for quasihyperbolic boundary condition domains in the case
$1\le q<\min\{n-nb,p\}$, $b=2\beta/(1+\beta)$. We use the upper
Minkowski dimension of the boundary. Roughly speaking, an issue is the
counting of the number of those Whitney cubes, of a given size, whose
shadows are comparable in measure. The shadow of a fixed Whitney cube
is the union of those cubes to which one goes through the fixed cube
when approaching the boundary of the domain from inside.  The use of
the upper Minkowski dimension enables us to count the aforementioned
cubes in an efficient manner. Previously the upper Minkowski dimension
of the boundary has been used in studying weighted Poincar\'e
inequalities in \cite{Br} and \cite{EH-S}, but maybe not to its full
potential. On the other hand, the upper Minkowski dimension seems to
be a right tool for the $(q,p)$-Poincar\'e inequality with $q < p$,
see Lemma~\ref{ballcount} and Lemma~\ref{sest} in
Section~\ref{sect:Poincare}.

More precisely, we show that a $\beta$-quasihyperbolic boundary
condition domain with the upper Minkowski dimension of the boundary
being less than or equal to $\lambda\in [n-1, n)$ supports the
$(q,p)$-Poincar\'e inequality \eqref{poincare} with $1\le q<p<\infty$
if
\begin{equation}\label{bound:qhbc}
p>\frac{q(n-\lambda b)}{q+b(n-\lambda)}, \qquad b=\frac{2\beta}{1+\beta},
\end{equation}
see Theorem~\ref{sharp}. The right hand side of the
inequality in \eqref{bound:qhbc} is, as it should be, an
increasing function of $\lambda$, when $q<n-n b$. Namely a
quasihyperbolic boundary condition domain is more irregular and
Poincar\'e inequality \eqref{poincare} fails to hold more easily when
the upper Minkowski dimension of the boundary is larger. We also show
that the bound in \eqref{bound:qhbc} is essentially sharp in
essentially all the possible cases in the plane, see
Theorem~\ref{sharp_counter_plane} and Remark \ref{rmk:sharpness1}; and
we discuss sharpness of the bound in higher dimensions, see
Theorem~\ref{sharp_counter}, Theorem~\ref{thm:modification}, and
Remark~\ref{dimension}.

To show that our results are sharp in the plane we introduce a method
for modifying any given John domain in a controlled manner so that the
resulting domain is no more a John domain but it satisfies a
quasihyperbolic boundary condition.

The structure of this paper is as follows. In
Section~\ref{sect:notation} we recall the quasihyperbolic boundary
condition and some basic facts related to this condition and the
geometry of Whitney cubes; we also recall the shadow of a Whitney
cube. Lemma~\ref{ballcount}, Lemma~\ref{log_shadow} and
Lemma~\ref{sest} in Section~\ref{sect:Poincare} are the key
ingredients in the proof of the main result of the paper,
Theorem~\ref{sharp}. In Section~\ref{bversion} we modify a given John
domain in order to revoke its John properties and to obtain a
quasihyperbolic boundary condition domain. We use such a modification
in Section~\ref{sharp_plane} where we consider sharpness of our main
result in the plane. We close the paper by giving an application to
the solvability of the Neumann problem on quasihyperbolic boundary
condition domains in Section~\ref{sect:appls}.

\section{Notation and preliminaries}
\label{sect:notation}

Throughout the paper $G$ is a bounded domain (an open connected set)
in $\R^n$, $n\ge 2$. The closure, the interior, and the boundary of a
set $E\subset\R^n$ are denoted by $\overline{E}$, $\inte(E)$, and
$\partial E$, respectively. We write $\chi_E$ for the characteristic
function of $E$, and the Lebesgue $n$-measure of a measurable set $E$
is written as $\vert E\vert$. The Hausdorff dimension is denoted by
$\mathrm{dim}_{\mathcal{H}}(E)$. The upper Minkowski dimension of a
set $E$ is
\[
\dim_\mathcal{M}(E)=\sup \big\{\lambda\ge 0\,:\, \limsup_{r\to 0+}
\mathcal{M}_\lambda(E,r)=\infty\big\},
\]
where for each $r>0$
\[
\mathcal{M}_\lambda(E,r)=
\frac{\left|\bigcup_{x\in E}B^n(x,r)\right|}{r^{n-\lambda}}
\]
is the $\lambda$-dimensional Minkowski precontent.

The family of closed dyadic cubes is denoted by $\mathcal{D}$. The
side length of a cube $Q\subset \R^n$ is $\ell (Q)$ and its centre is
$x_Q$. We let $\mathcal{D}_j$ be the family of those dyadic cubes
whose side length is $2^{-j}$, $j\in\Z$. For a domain $G$ we fix a
Whitney decomposition $\W=\W_G\subset\mathcal{D}$.  We write
$\mathcal{W}_j=\mathcal{W}\cap \mathcal{D}_j$, $j\in\Z$, and
$\sharp\W_j$ is the number of those cubes in $\W$ whose side length is
$2^{-j}$. For a Whitney cube $Q\in\mathcal{W}$ let us write
$Q^*=\frac{9}{8}Q$. Then
\begin{equation*}
\diam(Q)\le \dist(Q,\partial G)\le 4 \diam(Q),
\end{equation*}
and $\sum \chi_{Q^*}\le 12^n$. 
For the construction of Whitney cubes we refer to
Stein~\cite{S}.

Let us fix a cube $Q_0\in\W$. Then for each $Q\in \W$ there exists a
chain of cubes, $\C(Q):=(Q_0^*,Q_1^*,\ldots \,,Q_k^*)$, joining
$Q_0^*$ to $Q_k^*=Q^*$ such that $Q_i^*\cap Q_j^*\neq\emptyset\, $ if
and only if $\vert i-j\vert\le 1$. The length of this chain is
$\ell(\C(Q))=k$. Moreover, the shadow $S(Q)$ of a cube $Q\in\W$ is
defined as follows
\begin{equation}\label{shadow}
S(Q)=\bigcup_{\substack{R\in\mathcal{W}\\Q^*\in \C(R)}} R.
\end{equation}
The quasihyperbolic distance between points $x$ and $y$ in $G$ is
defined as
\[
k_G(x, y) = \inf_\gamma \int_\gamma \frac{ds}{\dist(z, \partial G)},
\]
where the infimum is taken over all rectifiable curves $\gamma$
joining $x$ to $y$ in $G$.  In this paper we study bounded domains
satisfying the following quasihyperbolic boundary condition. Other
equivalent definitions can be found, e.g., in \cite[p. 25]{H}.
\begin{defn}\label{qhbc}
  A bounded domain $G\subset\R^n$ is said to satisfy a {\em
    $\beta$-quasihyperbolic boundary condition}, $\beta \in (0,1]$, if
  there exist a point $x_0\in G$ and a constant $c<\infty$ such that
  \begin{equation*}
    k_G(x,x_0)\le \frac{1}{\beta}\log\frac{1}{\dist (x, \partial G)}+c
  \end{equation*}
holds for every $x\in G$.
\end{defn}
The following theorem is from \cite[Theorem 5.1]{KR}.
\begin{thm}\label{KR_theorem}
  Suppose $G$ satisfies a $\beta$-quasihyperbolic boundary condition.
  Then
$\mathrm{dim}_{\mathcal{M}}(\partial G)\le n-c_n\, \beta^{n-1}$
with a constant $c_n>0$ depending only on the dimension $n$.
\end{thm}
The notation $a\lesssim b$ means that an inequality $a\le cb$ holds
for some constant $c>0$ whose exact value is not important. We use
subscripts to indicate the dependence on parameters, for example,
$c_{\lambda}$ means that the constant depends only on the parameter
$\lambda$.

\section{Poincar\'e inequalities}
\label{sect:Poincare}

The following theorem is our main result.
\begin{thm}\label{sharp}
  Suppose $G$ satisfies a $\beta$-quasihyperbolic boundary condition,
  $\beta\in (0,1]$, and $\mathrm{dim}_{\mathcal{M}}(\partial G)\le
  \lambda\in [n-1,n)$.  If $1\le q<p<\infty$ are real numbers such that
\begin{equation}\label{oletus}
p>\frac{q(n-\lambda b)}{q+b(n-\lambda)},\qquad b=\frac{2\beta}{1+\beta},
\end{equation}
then $G$ supports the $(q,p)$-Poincar\'e inequality \eqref{poincare}.
\end{thm}

\begin{rem}
  Theorem \ref{sharp} is concerned with the case when the upper
  Minkowski dimension is bounded by $\lambda$.  Observe that the right
  hand side of the inequality in \eqref{oletus} is an increasing
  function of $\lambda$, when $q< n-n b$; recall from the introduction
  that this is the interesting case. This reflects the fact that a
  quasihyperbolic boundary condition domain is more irregular and the
  Poincar\'e inequality fails to hold more easily when the upper
  Minkowski dimension of the boundary is larger.
\end{rem}

\noindent{\em Preparations for the proof of Theorem \ref{sharp}}. Let
us choose $\lambda'\in (\lambda,n)$ such that inequality
\eqref{oletus} holds if $\lambda$ is replaced by $\lambda'$. Then
$\mathrm{dim}_{\mathcal{M}}(\partial G)<\lambda'$ and we may assume
that $\mathrm{dim}_{\mathcal{M}}(\partial G)$ is strictly less than
$\lambda\in [n-1,n)$. The following lemma from \cite[Lemma 4.4]{HH-SV}
relies on this strict inequality.
\begin{lem}\label{ballcount}
Let $K\subset \R^n$ be a compact set such that 
\[
\mathrm{dim}_{\mathcal{M}}(K)< \lambda
\]
where $\lambda\in [n-1,n)$. Assume that $\{B_1,B_2,\ldots,B_N\}$ is a
family of $N$ disjoint balls in $\R^n$, each of which is centered in
$K$ and whose radius is $r\in (0,1]$.  Then 
\[N\le cr^{-\lambda},
\]
where the constant $c$ is independent of the disjoint balls.
\end{lem}

Let $Q_0\in\W$ and $x_0\in Q_0$ be fixed.  Choose any $Q\in \W$
and join $x_0$ to $x_Q$ by a quasihyperbolic geodesic. By using those
Whitney cubes that intersect the quasihyperbolic geodesic we find, as
in \cite[Proposition~6.1]{H}, a chain $\C(Q)$ connecting $Q_0$ to $Q$
such that
\begin{equation}\label{equ:H6.1}
\ell(\C(Q)) \le c_n k_G(x_0, x) + 1\le 5c_n\big(\ell(\C(Q))+1\big),
\end{equation}
for every $x\in Q$. 

\begin{lem}\label{log_shadow}
  Suppose $G$ satisfies a $\beta$-quasihyperbolic boundary condition,
  $\beta \in (0,1]$.  Let $\varepsilon\in (0,1)$ and $1\le q
  <\infty$. Then
\begin{equation}\label{eps_est}
\sum_{\substack{Q\in\mathcal{W}\\Q\subset S(A)}}  \ell(\C(Q))^{q-1} |Q| \le c|\,S(A)\,|^{1-\varepsilon},
\end{equation}
where $c$ is a positive constant, independent of $A\in\W$.
\end{lem}

\begin{proof}
By inequality \eqref{equ:H6.1}
\begin{equation}\label{split_2}
\Sigma:=\sum_{\substack{Q\in\mathcal{W}\\Q\subset S(A)}}  \ell(\C(Q))^{q-1} |Q| \lesssim
\sum_{\substack{Q\in\mathcal{W}\\Q\subset S(A)}} \big(k_G(x_0, x_Q) + 1\big)^{q-1} |Q|.
\end{equation}
We employ H\"older's inequality with $r\in (1,\infty)$ and
$r'=r/(r-1)$ and estimate as follows
\[
\begin{split}
\Sigma 
&\lesssim
\bigg( \sum_{\substack{Q\in\mathcal{W}\\Q\subset S(A)}} \big(k_G(x_0, x_Q)+1\big)^{r'(q-1)} |Q|^{(1-1/r) r'} \bigg)^{1/r'}  \bigg( \sum_{\substack{Q\in\mathcal{W}\\Q\subset S(A)}} |Q|\bigg)^{1/r}\\
&\le 
\bigg( \sum_{\substack{Q\in\mathcal{W}}} 
\big(k_G(x_0, x_Q)+1\big)^{r'(q-1)}
 |Q| \,\bigg)^{1/r'}  |\,S(A)\,|^{1/r}.
\end{split}
\]
By inequality \eqref{equ:H6.1} and \cite[Theorem~7.7]{H} the last
series is finite, and its least upper bound depends only on $n$, $q$,
$r$, $x_0$, and $G$. Indeed, by \cite[Corollary 1]{SS}, domain $G$
satisfies the required Whitney-$\sharp$ condition.

The above estimates give 
\[\Sigma 
\le c|\,S(A)\,|^{1/r},\] where the constant $c$ depends only on $q$,
$n$, $r$, $x_0$, and $G$.  Choosing $r=1/(1-\varepsilon)>1$ gives
inequality \eqref{eps_est}.
\end{proof}

We write $[s]$ for the integer part of $s\in\R$.
\begin{lem}\label{sest}
  Suppose that $G$ satisfies a $\beta$-quasihyperbolic boundary condition,
  $\beta\in (0,1]$, and denote $b=2\beta/(1+\beta)$.  Suppose further that
  $\mathrm{dim}_{\mathcal{M}}(\partial G)< \lambda$, where $\lambda
  \in [n-1,n)$.  Then there is a number $\sigma\ge 1$ such that
\begin{equation}\label{kirjoitus}
\mathcal{W}_j= \bigcup_{k=0}^{[j-jb]} \mathcal{W}_{j,k,\sigma}
\end{equation}
for every $j\in\N$,
where 
\[
\mathcal{W}_{j,k,\sigma}:=
\{Q\in\mathcal{W}_j\,:\,2^{-(j-k)n}\le |S(Q)|\le \sigma 2^{-(j-k-1)n}\}.
\]
Also, if $k\in \{0,1,\ldots,[j-jb]\}$, then
\begin{equation}\label{sid}
\sharp\mathcal{W}_{j,k,\sigma}\le 
c j\,2^{n(j-k)+jb(\lambda-n)}.
\end{equation}
Here $c$ is a positive constant independent of $j$ and $k$.
\end{lem}

\begin{proof}
  Let us fix $j\in\N$. The $5r$-covering theorem, \cite[p. 23]{Mat95},
  implies that there is a finite family
\[\mathcal{F}\subset \{B^n(x,2^{-jb})\,:\,x\in\partial G\}\] of disjoint balls
such that
\begin{equation}\label{cover}
\partial G\subset \bigcup_{B\in\mathcal{F}} 5B.
\end{equation}
We claim that, if $Q\in\mathcal{W}_{j}$, there exists a ball
$B\in\mathcal{F}$ such that
\begin{equation}\label{ball_incl}
Q\subset c_1 B.
\end{equation} 
Here $c_1$ is a constant depending on $n$ only.  To verify this let
$y\in \partial G$ be a nearest point in $\partial G$ to the centre
$x_Q$ of $Q$. By inclusion \eqref{cover} there is a point $x$ in
$\partial G$ such that $B^n(x,2^{-jb})\in\mathcal{F}$ and $y\in
B^n(x,5\cdot 2^{-jb})$. If $z\in Q$ the triangle inequality
implies
\[
|z-x|\le |z-x_Q|+|x_Q-y|+|y-x|\le c2^{-j}+c2^{-j}+5\cdot 2^{-jb}<
c_12^{-jb}.
\]
Inclusion \eqref{ball_incl} follows because $Q\subset
B^n(x,c_1\,2^{-jb})=c_1\,B^n(x,2^{-jb})$.

Let us fix $Q\in\mathcal{W}_j$ and a ball $B:=B^n(x,2^{-jb})$ in
$\mathcal{F}$ such that $Q\subset c_1B$.  We claim that
\begin{equation}\label{cont}
S(Q)\subset B^n(x,c_2\,2^{-jb}),
\end{equation}
where $c_2>c_1$ is a constant which depends on $\beta$, $n$, $x_0$,
and $G$ only.  To prove inclusion \eqref{cont}, let $R\in\mathcal{W}$
be a cube such that $Q^*\in \C(R)$. Then $R\subset S(Q)$. By
\cite[Lemma 6]{JK}, 
\begin{align*}
\mathrm{diam}(R)\le \diam(S(Q))&
\le c_{\beta,n,x_0,G}\diam(Q)^{2\beta/(1+\beta)}
= c2^{-jb}.
\end{align*}
Then, if $y\in R$, the triangle inequality gives
\begin{align*}
  |y-x| &\le |y-x_R|+|x_R-x_Q|+|x_Q-x| \le c2^{-jb} +
  c2^{-jb} + c_12^{-jb} < c_2 2^{-jb}.
\end{align*}
Inclusion \eqref{cont} follows.

As a consequence of \eqref{cont}, we obtain
\[
2^{-jn}=|Q|\le |S(Q)|\le \sigma 2^{-jn b}
\]
with a constant $\sigma\ge 1$, depending on $\beta$, $n$, $x_0$, and
$G$ only. Identity \eqref{kirjoitus} is valid with this constant.

To prove estimate \eqref{sid} we use the following inequality
\begin{equation}\label{eks}
  \sharp \{Q\in\mathcal{W}_{j}\,:\,Q^*\in \C(R)\}\le c_{\beta,n,x_0,G}j,\qquad R\in\mathcal{W},
\end{equation}
from \cite[Lemma 2.5]{KOT}.

Let us fix $k\in \{0,1,\ldots,[j-jb]\}$ and let us consider an
arbitrary ball $B:=B^n(x,2^{-jb})$ in $\mathcal{F}$.  We estimate
the number of cubes that are included in $c_1B$.  By inclusion
\eqref{cont}
\begin{align*}
  &\sharp\{Q\in\mathcal{W}_{j,k,\sigma}\,:\, Q\subset c_1 B\} \\&\le
  \sum_{\substack{Q\in\mathcal{W}_{j,k,\sigma}\\Q\subset c_1B}}
  2^{(j-k)n}|S(Q)|
  \le 2^{(j-k)n}\sum_{\substack{Q\in\mathcal{W}_{j,k,\sigma}}} |S(Q)\cap c_2B|\\
  &\le
  2^{(j-k)n}\sum_{\substack{Q\in\mathcal{W}_{j,k,\sigma}}}\sum_{\substack{R\in\mathcal{W}
      \\ Q^*\in \C(R)}} |R\cap c_2B| =
  2^{(j-k)n}\sum_{R\in\mathcal{W}}
  \sum_{\substack{Q\in\mathcal{W}_{j,k,\sigma} \\ Q^*\in \C(R)}}
  |R\cap c_2B|.
\end{align*}
Inequality \eqref{eks} shows that the last sum is bounded by
\[
c_{\beta,n,x_0,G} j 2^{(j-k)n}|c_2 B|\le c_3 j 2^{-kn}2^{j(n-nb)}
\]
with a constant $c_3>0$ which depends on $\beta$, $n$, $x_0$, and $G$ only.

Inclusion \eqref{ball_incl} implies that
\begin{equation}\label{cos}
  \sharp \mathcal{W}_{j,k,\sigma}  \le
  \sum_{B\in\mathcal{F}}  \sharp\{Q\in\mathcal{W}_{j,k,\sigma}\,:\,Q\subset c_1B\}\le c_3\sum_{B\in\mathcal{F}} j\,2^{-kn}2^{j(n-nb)}.
\end{equation}
Recall that $\mathcal{F}$ is a family of disjoint balls, each of which
are centered in $\partial G$ with radius $2^{-jb}\in (0,1]$.
Therefore, Lemma~\ref{ballcount} yields
\[\sharp\mathcal{F}\le c2^{j\lambda b}.\]
Combining this estimate with inequalities \eqref{cos} yields
\[
\sharp \mathcal{W}_{j,k,\sigma}\le cj\,2^{j\lambda b }2^{-kn}2^{j(n-nb)},
\]
which is estimate \eqref{sid}.
\end{proof}

\begin{proof}[Proof of Theorem \ref{sharp}]
  We may assume, by scaling, that $\diam (G)<1$. Hence
  $\W=\bigcup_{j=0}^\infty\W_j$.  Using H\"older's inequality, and
  inequalities $|a+b|^q\le 2^{q-1}(|a|^q+|b|^q)$ and $|a+b|^{1/q}\le
  |a|^{1/q}+|b|^{1/q}$, $a,b\in\R$, we obtain
  \begin{align} \label{eq:PI} & \bigg(\int_G|u(x)-u_G|^q \,
      dx\bigg)^{1/q} \leq 2 \bigg(\int_G|u(x)-u_{Q_0^*}|^q\, dx\bigg)^{1/q}  \nonumber \\
    & \quad \lesssim \bigg(\sum_{Q\in\W}\int_{Q^*}|u(x)-u_{Q^*}|^q\, dx\bigg)^{1/q}
    + \bigg(\sum_{Q\in\W}\int_{Q^*}|u_{Q^*}-u_{Q_0^*}|^q\,
      dx\bigg)^{1/q}.
\end{align}
The first term on the right hand side of \eqref{eq:PI} is estimated by
the $(q,p)$-Poincar\'e inequality in cubes and by H\"older's
inequality,
\[
\bigg(\sum_{Q\in\W}\int_{Q^*}|u(x)-u_{Q^*}|^q\, dx\bigg)^{1/q}
\lesssim \bigg(\int_G|\nabla u(x)|^p\, dx\bigg)^{1/p}.
\]
For the second term on the right hand side of \eqref{eq:PI} let us
connect $Q_0$ to every cube $Q\in\W$ by a chain $\C(Q)=
(Q_0^*,Q_1^*,\ldots \,,Q_k^*)$, $Q_k^*=Q^*$, that is constructed by
using quasihyperbolic geodesics as in connection with
\eqref{equ:H6.1}. By the triangle and H\"older's inequalities, by the
properties of Whitney cubes and by the validity of $(q,p)$-Poincar\'e
inequality in cubes we obtain
\begin{align*}
  &\sum_{Q\in\W}\int_{Q^*} |u_{Q^*}-u_{Q_0^*}|^q\, dx\lesssim
  \sum_{Q\in\W}\int_{Q^*} 
\ell(\C(Q))^{q-1}
  \sum_{j=0}^{k-1}|u_{Q_{j}^*}-u_{Q_{j+1}^*}|^q\, dx\\
  &\lesssim \sum_{Q\in\W}\int_{Q^*} 
  \ell(\C(Q))^{q-1}
\sum_{j=0}^{k}
  |Q_{j}^*|^{-1}\int_{Q_{j}^*}
  |u(y)-u_{Q_{j}^*}|^q\,dy\, dx \\
  & \lesssim\sum_{Q\in\W}\int_{Q^*} 
  \ell(\C(Q))^{q-1}
  \sum_{j=0}^{k}
  |Q_{j}^*|^{q/n-q/p}\bigg(\int_{Q_{j}^*} |\nabla u(y)|^p\,dy\bigg)^{q/p}\, dx\,.
\end{align*}
By rearranging the double sum and by using H\"older's inequality with
$(p/q, p/(p-q))$ we obtain
\begin{align*}
  &\sum_{Q\in\W}\int_{Q^*} |u_{Q^*}-u_{Q_0^*}|^q\, dx \\
  & \quad \lesssim \sum_{A\in \W} \sum_{\substack{Q\in \W\\
      Q\subset S(A)}}\ell(\C(Q))^{q-1}|Q||A|^{q/n-q/p}
  \biggl(\int_{A^*}|\nabla u(x)|^p\,dx\biggr)^{q/p} \\
  & \quad \lesssim \bigg( \sum_{A\in \W} \biggl(\sum_{\substack{Q\in \W\\
      Q\subset
      S(A)}}\ell(\C(Q))^{q-1}|Q||A|^{q/n-q/p}\biggr)^{p/(p-q)}\bigg)^{(p-q)/p}
  \bigg(\int_{G}|\nabla u(x)|^p\, dx\bigg)^{q/p}\,.
\end{align*}
We write 
\[
\Sigma:=\sum_{A\in \W} \bigg(\sum_{\substack{Q\in \W\\ Q\subset
     S(A)}}\ell(\C(Q))^{q-1}|Q||A|^{q/n-q/p}\bigg)^{p/(p-q)}.
\]
Hence it is enough to show that the quantity $\Sigma$ is finite.
The preceding part of the proof followed the chaining argument in
\cite[Theorem 3.2]{HH-SV}, which is nowadays a standard approach
dating back to \cite{J}.  

For the following computation, it is convenient to denote
$b=2\beta/(1+\beta)$. Fix $\varepsilon\in (0,q/p)$. By
Lemma~\ref{log_shadow},
\[
\Sigma\lesssim \sum_{A\in \mathcal{W}} \big( |\,S(A)\,|^{1-\varepsilon}\, |A|^{q/n-q/p}\big)^{p/(p-q)}.
\]
By \eqref{kirjoitus} we obtain
\begin{align*}
  \Sigma &\lesssim \sum_{j=0}^\infty
  \sum_{k=0}^{[j-jb]}\sum_{A\in\mathcal{W}_{j,k,\sigma}} \big(
  |S(A)|^{1-\varepsilon}\, |A|^{q/n-q/p}\big)^{p/(p-q)}.
\end{align*}   Definition of
$\mathcal{W}_{j,k,\sigma}$ and inequality \eqref{sid} imply
\begin{align*}
\Sigma&\lesssim \sum_{j=0}^\infty \sum_{k=0}^{[j-jb]}  j2^{n(j-k)+jb(\lambda-n)}\big( 2^{-n(j-k)(1-\varepsilon)}2^{-jnq(1/n-1/p)}\big)^{p/(p-q)}\\
&= \sum_{j=0}^\infty \sum_{k=0}^{[j-jb]}  j2^{kn(p(1-\varepsilon)/(p-q)-1)} 
2^{j(n+b(\lambda-n)-n(1-\varepsilon)p/(p-q)-qp/(p-q)+nq/(p-q))}.
\end{align*}
Let $j$ and $k$ be as in the summation. Then,
\[
kn\bigg(\frac{p(1-\varepsilon)}{p-q}-1\bigg) \le n(j-jb)\bigg(\frac{p(1-\varepsilon)}{p-q}-1\bigg)
=\frac{jn(1-b)(q-\varepsilon p)}{p-q}.
\]
By the estimate $[j-jb]\le j$, where $b\in (0,1]$,
\begin{align*}
  \Sigma &\lesssim \sum_{j=0}^\infty j^2 2^{j(n(1-b)(q-\varepsilon p)/(p-q)+n+b(\lambda-n)-n(1-\varepsilon)p/(p-q)-qp/(p-q)+nq/(p-q))}\\
  &\lesssim \sum_{j=0}^\infty j^2 2^{\varepsilon
    j(np/(p-q)-n(1-b)p/(p-q))} 2^{j(b(\lambda-n) -q((b-1)n+p)/(p-q))}.
\end{align*}
Inequality \eqref{oletus} allows us to choose $\varepsilon>0$ so small that
the last series converges.
\end{proof}

We have the following corollary of Theorem~\ref{sharp}.

\begin{cor}\label{k}
  Suppose $1\le q<n-nb$, where $b=2\beta/(1+\beta)$ and $\beta \in
  (0,1)$.  If $G$ satisfies a $\beta$-quasihyperbolic boundary
  condition, then $G$ supports the $(q,n-nb)$-Poincar\'e inequality
  \eqref{poincare}.
\end{cor}

\begin{proof}
By Theorem \ref{KR_theorem}, $\lambda=\mathrm{dim}_{\mathcal{M}}(\partial G)<n$.
Observe that $r=n-nb$ satisfies
the identity
\[
p(r)= \frac{r(n-\lambda b)}{r+b(n-\lambda)} = r.
\]
On the other hand, since $1\le q<r$ and
$p'(t)>0$ if $t>0$, 
we obtain that $p(q)<p(r)=r=n-n b$. The
claim follows from Theorem \ref{sharp}.
\end{proof}

\section{Modification of a John domain}
\label{bversion}

In this section, we introduce a method how to modify a given John
domain $G$ in a controlled way such that the resulting domain, denoted
by $G_{\beta}$ and called a $\beta$-version of $G$, is no more a John
domain but it satisfies a $\beta/4$-quasihyperbolic boundary condition
if $\beta\le \varkappa c_J$. Here $\varkappa$ is a constant depending
on $n$ only, and $c_J$ is the John constant of $G$ (see
Definition~\ref{sjohn}). By studying the validity of Poincar\'e
inequalities \eqref{poincare} on these $\beta$-versions of John
domains we shall show that Theorem~\ref{sharp} is essentially sharp in
the plane. 
Theorems \ref{sharp_counter} and \ref{thm:modification} 
are the main results of this section.

We assume that $G$ is a bounded domain such that $\diam(G)\le 4$ in
Section~\ref{bversion} and Section~\ref{sharp_plane}. We recall the
following definition.
\begin{defn}\label{sjohn}
  A bounded domain $G\subset\R^n$ is a {\em John domain}, if there
  exist a point $x_0$ in $G$ and a constant $c_J\in (0,1]$ such that
  every point $x\in G$ can be joined to $x_0$ by a rectifiable curve
  $\gamma:[0,\ell(\gamma)]\to G$ which is parametrized by its arc length,
  $\gamma(0)=x$, $\gamma(\ell(\gamma))=x_0$, and
\begin{equation}\label{eq:John}
\dist(\gamma(t),\partial G)\ge c_J t
\end{equation}
for each $t\in [0,\ell(\gamma)]$. The point $x_0$ is called a {\em John center}
of $G$, and the largest $c_J\in (0,1]$ is called the {\em John
  constant} of $G$.  The curve $\gamma$ is called a {\em John curve}.
\end{defn}

Let us fix $\beta \in (0,1]$ and let $Q\subset\R^n$ be a closed cube
that is centered at $x_Q=(x_1,\ldots,x_n)$, and whose side length is
$\ell(Q)=\ell\le 4$, i.e. 
\[
Q:=\prod_{i=1}^n\,[x_i-\ell/2,\, x_i+\ell/2].
\]
The {\em room} in $Q$ is the open cube
\[
R(Q):=\inte\left(\frac{1}{4}Q\right)=\prod_{i=1}^{n} (x_i-\ell/8,\, x_i+\ell/8)
\]
whose center is $x_Q$ and side length is $\ell/4$.
The {\em $\beta$-passage} in $Q$ is the open set
\[
P_\beta(Q):= \bigg(\prod_{i=1}^{n-1} \big(x_i-(\ell/8)^{1/\beta},\,
x_i+(\ell/8)^{1/\beta} \big)\bigg) \times (x_n+\ell/8,\,
x_n+\ell/8+(\ell/8)^{1/\beta}).
\]
Note that $\ell/8< 1$ and $1/\beta\ge 1$, hence we have
$(\ell/8)^{1/\beta} < \ell/8$. Thus, $P_\beta(Q)\subset \frac{1}{2}Q$.
The open cube
\[
E(Q):=\mathrm{int}\left(\frac{3}{4}Q\right)=\prod_{i=1}^{n} (x_i-3\ell/8,\,
  x_i+3\ell/8)\subset Q
\]
contains the room and $\beta$-passage in $Q$.  The {\em long
  $\beta$-passage} in $Q$ is the open set
\[
L_\beta(Q):= \bigg(\prod_{i=1}^{n-1} \big(x_i-(\ell/8)^{1/\beta},\,
x_i+(\ell/8)^{1/\beta} \big)\bigg) \times (x_n,\, x_n+\ell/2)\subset
Q.
\]
The {\em $\beta$-apartment} in $Q$ is the set
\begin{equation*}
A_\beta(Q):= L_\beta(Q)\cup Q\setminus (\partial R(Q)\cup \partial P_\beta(Q))\subset Q.
\end{equation*}
Figure~\ref{fig:Q} depicts these geometric objects in a cube $Q$ when
$\beta=1/2$.

\begin{figure}[!htbp] 
\centering
\includegraphics[width=0.45\textwidth]{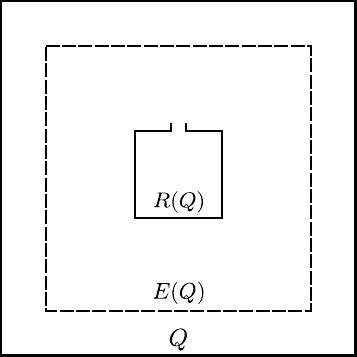}
\caption[Figure 11]{A modified Whitney cube with $\beta=1/2$.}
\label{fig:Q}
\end{figure}

\begin{defn}\label{s_version}
  Let $G$ be a John domain. A {\em $\beta$-version of $G$} is the
  domain
\[
G_\beta := \bigcup_{\substack{Q\in\mathcal{W}_G}} A_\beta(Q).
\]
\end{defn}
The following proposition is a modification of \cite[Proposition
5.11]{HH-SV}.
\begin{prop}\label{dimpres}
  Let $G\subset\R^n$ be a John domain. Then
\[\dim_\mathcal{M}(\partial G)=\dim_\mathcal{M}(\partial G_\beta)\]
for every $\beta\in (0,1]$.
\end{prop}

Let us next study the validity of Poincar\'e inequalities on a
$\beta$-version of a given John domain. Let $Q\subset \R^n$ be a
closed cube, $\ell(Q)\le 4$, and define a continuous function
\[
u^{A_\beta(Q)}\colon G_\beta\to \R
\]
which has linear decay along the $n^{\text{\tiny th}}$  variable in $P_\beta(Q)$
and satisfies
\begin{equation}\label{values}
u^{A_\beta(Q)}(x) =\begin{cases}
\ell(Q)^{(\lambda-n)/q},\qquad &\text{ if }x\in R(Q);\\
0,\qquad &\text{ if } x\in G_\beta\setminus (R(Q)\cup \overline{P_\beta(Q)}).
\end{cases}
\end{equation}
Moreover, in the sense of distributions in $G_\beta$ the following
holds
\begin{equation}\label{gradient}
\nabla u^{A_\beta(Q)} = (0,\ldots,0,-8^{1/\beta}\ell(Q)^{(\lambda-n)/q-1/\beta}\chi_{P_\beta(Q)}).
\end{equation}

The following is the first main result in this section.

\begin{thm}\label{sharp_counter}
  Let $G\subset\R^n$ be a John domain such that
\begin{equation} \label{eq:sharp-counter}
\limsup_{k\to \infty} 2^{-\lambda k}\sharp \mathcal{W}_k>0,
\end{equation}
where $\lambda=\mathrm{dim}_{\mathcal{M}}(\partial G)$. Suppose that
\[
p\le \frac{q(n-\lambda b)}{q+b(n-\lambda)},\qquad b=\frac{2\beta}{1+\beta},
\]
where $1\le q<p<\infty$ and $\beta\in (0,1]$.  Then the
$\beta$-version of $G$ does not satisfy the $(q,p)$-Poincar\'e
inequality \eqref{poincare}.
\end{thm}

By \cite[Proposition 5.2]{HH-SV} for every $\lambda\in[n-1,n)$, $n\geq
2$, there exists a John domain $G\subset\R^n$ such that
$\mathrm{dim}_{\mathcal{M}}(\partial G)=\lambda$ and
hypothesis~\eqref{eq:sharp-counter} in Theorem~\ref{sharp_counter} is
satisfied.

\begin{proof}[Proof of Theorem~\ref{sharp_counter}] 
By the assumptions and inequality $\beta \le b\le 1$, 
\begin{equation}\label{rels}
p\le \frac{q(n-\lambda\beta)}{q+\beta(n-\lambda)}\,.
\end{equation} 
  Choose $k_0\in\N$ such that $\limsup_{k\to \infty}
  2^{-\lambda(k-k_0)}\sharp \mathcal{W}_k>2$. This allows us to
  inductively choose indices $j(k)$, $k\in \N$, such that
\[
k_0\le j(1)<j(2)<\dotsb\] and $\sharp\mathcal{W}_{j(k)} \ge 2\cdot
2^{\lambda (j(k)-k_0)}$ for every $k\in \N$.  Let us write
$M_j:=2^{[\lambda(j-k_0)]}$, where $[\lambda(j-k_0)]$ is the integer
part of $\lambda(j-k_0)$, and let us choose cubes
$Q_{j(k)}^1,\ldots,Q_{j(k)}^{2M_{j(k)}}\in \W_{j(k)}$.  For every
$m\in\N$ let us write
\[
v_m := \sum_{k=1}^m \bigg(\sum_{i=1}^{M_{j(k)}}
u^{A_\beta(Q_{j(k)}^i)}-\sum_{i=M_{j(k)}+1}^{2M_{j(k)}}
u^{A_\beta(Q_{j(k)}^i)}\bigg)\in W^{1,p}(G_\beta).
\]
Note that $(v_m)_{G_\beta}=0$ and
\begin{align*}
A_m:&=\bigg(\int_{G_\beta} |v_m - (v_m)_{G_\beta}|^q\bigg)^{1/q}
= \bigg(\sum_{k=1}^m\sum_{i=1}^{2M_{j(k)}} \int_{G_\beta}|u^{A_\beta(Q_{j(k)}^i)}(x)|^q\,dx\bigg)^{1/q}
\\&\ge 
\bigg(\sum_{k=1}^m 2\cdot 2^{\lambda(j(k)-k_0)-1}2^{-j(k)(\lambda-n)}4^{-n}2^{-j(k)n}\bigg)^{1/q} = 
c_{n,q,\lambda,k_0}m^{1/q}.
\end{align*}
On the other hand, by inequality \eqref{rels},
\begin{align*}
B_m: &=\bigg( \int_{G_\beta} |\nabla v_m(x)|^p\,dx\bigg)^{1/p} \\&=  \bigg(\sum_{k=1}^m\sum_{i=1}^{2M_{j(k)}} \int_{G_\beta}|\nabla u^{A_\beta(Q_{j(k)}^i)}(x)|^p\,dx\bigg)^{1/p}
\\&
\le c_\beta\bigg(\sum_{k=1}^m 2\cdot 2^{\lambda(j(k)-k_0)}2^{-pj(k)((\lambda-n)/q-1/\beta)}2^{n-1}2^{-nj(k)/\beta}\bigg)^{1/p}
\\& \le c_{n,\beta,p,\lambda,k_0}m^{1/p}.
\end{align*}
Hence, we obtain 
\[
\frac{A_m}{B_m} \ge c_{n,\beta,p,q,k_0,\lambda}m^{1/q-1/p}\xrightarrow{m\to \infty} \infty,
\]
where $1\le q<p$.
It follows that $G_\beta$ does not satisfy the $(q,p)$-Poincar\'e inequality
\eqref{poincare}.
\end{proof}

We shall show that a $\beta$-version of a given John domain $G$
satisfies a $\beta/4$-quasihyperbolic boundary condition if
$\beta\le\varkappa c_J$, where $c_j$ is the John constant of $G$ and
$\varkappa$ is a constant depending on $n$ only,
Theorem~\ref{thm:modification}.  Let us begin with the following
auxiliary result.

\begin{lem}\label{mod_lemma}
  Let $G\subset\R^n$ be a John domain with the John constant $c_J$.
  Let further $\beta\in (0,1)$ and $G_\beta$ be the $\beta$-version of
  $G$. If $x\in E(Q)\cap G_\beta$, $Q\in\mathcal{W}_G$, then there is
  a point $z\in \partial E(Q)$ such that
\begin{equation}\label{xz}
  k_{G_\beta}(x,z)\le 
  \bigg(\frac{2}{\beta}+\frac{1}{\varkappa} \bigg) \log \frac{1}{\dist(x,\partial G_\beta)} + c.
\end{equation}
If $x\in Q\setminus E(Q)$, then
\begin{equation}\label{xx0}
k_{G_\beta}(x,x_0)\le \frac{1}{\varkappa c_J}\log \frac{1}{\dist(x,\partial G_\beta)} + c.
\end{equation}
The constant $\varkappa\in(0,1)$ appearing in both inequalities~\eqref{xz} and \eqref{xx0}
depends on $n$ only.
\end{lem}

\begin{proof}
  Let us consider the case $x\in R(Q)\subset E(Q)$.  We will join $x$
  to the point $z:=x_Q+(3\ell/8) e_n\in \partial E(Q)$, where $e_n=(0,\ldots,0,1)$, by
  a rectifiable curve that is to be constructed next.
   For this
  purpose, we record the following inequality
\begin{equation}\label{x_est}
\dist(x,\partial G_\beta)\le \ell/8.
\end{equation}
Notice that there is a constant $\varkappa$ such that
\[
k_{G_\beta}(x,x_Q)\le \frac{1}{\varkappa} \log \frac{1}{\dist(x,\partial G_\beta)} +c.
\]
Next we connect $x_Q$ to the point $z$ by a curve
$\gamma:[0,3\ell/8]\to G_\beta$ which parametrizes the line segment
$[x_Q,z]$.  Write $\xi=\ell/8-(\ell/8)^{1/\beta}$ and $\eta = \ell/8+
2(\ell/8)^{1/\beta}$. Observe that
\[
\dist(\gamma(t),\partial G_\beta)\ge \ell/8 - t
\]
for all $t\in [0,\ell/8]$. By this inequality and \eqref{x_est}, we
obtain
\begin{align*}
\int_0^{\xi}\frac{dt}{\dist(\gamma(t),\partial G_\beta)}
\le \int_0^\xi \frac{dt}{\ell/8-t} 
&\le  \frac{1}{\beta} \log \frac{1}{\dist(x,\partial G_\beta)}.
\end{align*}
In the following step, we pass through the $\beta$-passage in $Q$. If
$t\in [\xi,\eta]$ then $\dist(\gamma(t),\partial G_\beta) \ge
(\ell/8)^{1/\beta}$. Hence
\begin{align*}
\int_\xi^{\eta} \frac{dt}{\dist(\gamma(t),\partial G_\beta)}\le 
\frac{3(\ell/8)^{1/\beta}}{(\ell/8)^{1/\beta}}=3.
\end{align*}
For $t\in [\eta,3\ell/8]$, 
\[
\dist(\gamma(t),\partial G_\beta)\ge \min\{\ell/8,(\ell/8)^{1/\beta}+
t - \eta\}.
\]
By this inequality and the fact that $({3\ell/8-\eta})/({\ell/8})\le 3$,
\begin{align*}
\int_{\eta}^{3\ell/8} \frac{dt}{\dist(\gamma(t),\partial G_\beta)} &\le 3+\int_\eta^{3\ell/8} \frac{dt}{(\ell/8)^{1/\beta}+  t - \eta}
\\&
\le \frac{1}{\beta} \log \frac{1}{\ell/8} + c
\le \frac{1}{\beta} \log\frac{1}{\dist(x,\partial G_\beta)} + c.
\end{align*}
By these estimates we have 
\[
k_{G_\beta} ( x,z) \le k_{G_\beta}(x,x_Q)+k_{G_\beta}(x_Q,z)\le \bigg(\frac{2}{\beta}+\frac{1}{\varkappa}\bigg) 
\log\frac{1}{\dist(x,\partial G_\beta)} + c
\]
for every $x\in R(Q)$. This gives inequality~\eqref{xz}.

Let us consider the case $x\in G_\beta\cap
\overline{P_\beta(Q)}$. There is a point $\omega$ on the line segment
from $x_Q$ to $z$ such that
\[
k_{G_\beta}(x,\omega)\le \frac{1}{\varkappa}\log\frac{1}{\dist(x,\partial G_\beta)} + c.
\]
Joining $\omega$ to $z$ by the line segment $[\omega,z]\subset
[x_Q,z]$ gives the inequality
\[
k_{G_\beta}(\omega,z)
\le \frac{1}{\beta} \log \frac{1}{\ell/8}+c.
\]
Since $\dist(x,\partial G_\beta)\le (\ell/8)^{1/\beta}\le \ell/8$, we
have
\[
k_{G_\beta}(x,z)\le k_{G_\beta}(x,\omega)+k_{G_\beta}(\omega,z)\le
\bigg(\frac{1}{\beta}+\frac{1}{\varkappa}\bigg) \log\frac{1}{\dist(x,\partial G_\beta)} + c.
\]
This gives inequality~\eqref{xz}.

If $x\in E(Q)\setminus (R(Q)\cup \overline{P_\beta(Q)})$, then clearly
there is a point $z\in \partial E(Q)$ such that
\[
k_{G_\beta}(x,z)\le \frac{1}{\varkappa}\log\frac{1}{\dist(x,\partial G_\beta)} + c.
\]
But this is inequality~\eqref{xz}.

Finally, let us consider the case $x\in Q\setminus E(Q)$.  The idea is
to construct a curve $\gamma:[0,\ell(\gamma)]\to G_\beta$,
parametrized by arc length such that $\gamma(0)=x$ and
$\gamma(\ell(\gamma))=x_0$, so that
\begin{equation}\label{des}
\dist(\gamma(t),\partial G_\beta) \ge \varkappa c_J t
\end{equation}
for every $t\in [0,\ell(\gamma)]$. This is done by taking a John curve
from $x$ to the John center $x_0$, and modifying it whenever it
intersects with $E(Q)$, $Q\in\mathcal{W}_G$. This is illustrated in
Figure~\ref{fig:Curve}.  For further details, we refer to the proof of
\cite[Proposition 5.16]{HH-SV}.

Let us write $\delta=\dist(x,\partial G_\beta)$. If $\ell(\gamma)\le
\delta/2$, then $k_{G_\beta}(x,x_0)\le 1$. Hence, we may assume that
$\ell(\gamma)>\delta/2$ and therefore
\begin{align*}
  k_{G_\beta} (x,x_0) &\le \int_{\gamma}\frac{ds}{\dist(z,\partial G_\beta)}\\
  &=\int_0^{\delta/2} \frac{dt}{\dist(\gamma(t),\partial G_\beta)}+
  \int_{\delta/2}^{\ell(\gamma)} \frac{dt}{\dist(\gamma(t),\partial
    G_\beta)}.
\end{align*}
If $t\in [0,\delta/2]$ then $\dist(\gamma(t),\partial G_\beta)\ge
\delta/2$. Inequality~\eqref{des} implies that $\ell(\gamma)\le
\diam(G_\beta)/\varkappa c_J=:T$. Hence,
\begin{align*}
k_{G_\beta}(x,x_0)\le 1+ \frac{1}{\varkappa c_J}\int_{\delta/2}^{T} \frac{dt}{t}
&\le 1+\frac{\log(T)-\log(\delta/2)}{\varkappa c_J}\\
&= \frac{1}{\varkappa c_J} \log\frac{1}{\dist(x,\partial G_\beta)} + c.
\end{align*}
This gives inequality~\eqref{xx0}.
\end{proof}

\begin{figure}[!htbp]
\centering
\includegraphics[width=0.5\textwidth]{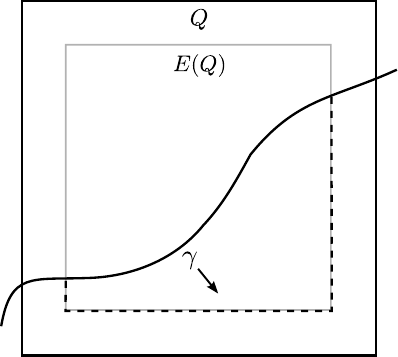}
\caption{John curve $\gamma$ and its modification (dotted line) on
  $E(Q)$.}
 \label{fig:Curve}
\end{figure}

The following is the second main result in this section.

\begin{thm}\label{thm:modification}
Suppose $G\subset\R^n$ is a John domain with a John constant $c_J$.
  Then the $\beta$-version of $G$, $G_\beta$, satisfies a
  $\beta/4$-quasihyperbolic boundary condition if $0<\beta \le
  \varkappa c_J$, where $\varkappa\in (0,1)$ is the same constant as
  in Lemma~\ref{mod_lemma}.
\end{thm}

\begin{proof}
Fix a point $x$ in $G_{\beta}\subset G$.  We shall show that the inequality
\begin{equation}\label{wanted}
k_{G_{\beta}}(x,x_0) \leq
\frac{4}{\beta}\log\frac{1}{\dist(x,\partial G_{\beta})} + c
\end{equation}
is valid with some constant $c$; the point $x$ belongs to some Whitney
cube $Q\in\mathcal{W}_G$.

Let us consider the case $x\in E(Q)\cap G_\beta$. By Lemma~\ref{mod_lemma},
there is a point $z\in \partial E(Q)$ such that inequality~\eqref{xz}
holds. Applying Lemma~\ref{mod_lemma} again, now with the point $z\in
Q\setminus E(Q)$, yields
\begin{align*}
  k_{G_\beta}(x,x_0)&\le k_{G_\beta}(x,z)+k_{G_\beta}(z,x_0) \\&\le
  \bigg(\frac{2}{\beta}+\frac{1}{\varkappa}\bigg)\log
  \frac{1}{\dist(x,\partial G_\beta)} + \frac{1}{\varkappa c_J} \log
  \frac{1}{\dist(z,\partial G_\beta)} +c.
\end{align*}
By inequality $\dist(z,\partial G_\beta)\ge \ell/8\ge \dist(x,\partial G_\beta)/c_n$,
\[
k_{G_\beta}(x,x_0)\le \frac{4}{\beta}\log \frac{1}{\dist(x,\partial G_\beta)}+c.
\]
This gives inequality~\eqref{wanted}.

If $x\in Q\setminus E(Q)$, then inequality \eqref{wanted} follows from
Lemma~\ref{mod_lemma}.
\end{proof}

\section{Sharpness of Theorem \ref{sharp} in the plane}
\label{sharp_plane}

We consider sharpness of our main result, Theorem \ref{sharp}, in the
plane. We introduce new constants $\tau$ and $\bar c_2$ whose values
will be clear later. It will become apparent that we would like to
have the constant $\bar c_2$ as close as possible to the constant
$c_2$ in Theorem~\ref{KR_theorem}. The following theorem is a
delicate extension of Theorem~\ref{sharp_counter} in the plane.

\begin{thm}\label{sharp_counter_plane}
  Let $\lambda\in [1,2)$ and let $\beta\in (0,1)$ be such that
  $\lambda \le 2-{\bar c_2}\, \beta$. Then there is a domain
  $G_\beta\subset\R^2$ which satisfies a $\beta/4$-quasihyperbolic
  boundary condition, $\dim_{\mathcal{M}}(\partial G_\beta)=\lambda$,
  and which does not satisfy the $(q,p)$-Poincar\'e inequality
  \eqref{poincare} if $1\le q<p<\infty$, and 
\[
p\le \frac{q(2-\lambda b)}{q+b(2-\lambda)},\qquad
b=\frac{2\beta}{1+\beta}.
\]
\end{thm}

\begin{rem}\label{dimension}
  A counterpart of Theorem~\ref{sharp_counter_plane} for $n\ge 3$ is
  not known to us. It seems that the construction behind
  Theorem~\ref{sharp_counter_plane} can be generalized only if
  $\lambda \le n-\bar c_n\,\beta$. However, Theorem~\ref{KR_theorem}
  suggests that it is natural to consider the less restrictive
  condition, $\lambda \le n- \bar c_n\,\beta^{n-1}$, which allows
  larger values of $\beta$, i.e. more cases to be covered by a
  counterexample, for a a given $\lambda$. An obstacle is that it
  seems not to be known whether Theorem~\ref{KR_theorem} is sharp 
  in case of
  $n\ge 3$.
\end{rem}

We need the following construction for the proof of
Theorem~\ref{sharp_counter_plane}.

\begin{prop}\label{olemassa}
  Let $\lambda\in [1,2)$ and $\beta\in (0,1)$ be such that $\lambda
  \le 2-\tau \beta$. Then there exists a John domain $G\subset\R^2$
  with a John constant $c_J\ge \beta$ such that
  $\dim_\mathcal{M}(\partial G)=\lambda$ and
\begin{equation}\label{paljon}
  \limsup_{k\to \infty} 2^{-\lambda k}\sharp \mathcal{W}_k
  >0.
\end{equation}
\end{prop}

\begin{proof}
  Let $Q:=[-1,1]\times [-1,1]\subset \R^2$ and, for $\kappa\in (0,1)$,
  $r(\kappa):=(1-\kappa)/2\in (0,1/2)$.  Let us write
  $z_1:=(\kappa+r(\kappa),\kappa+r(\kappa))$, and let $z_2,z_3,z_4$ be
  the corresponding symmetric points in the three remaining quadrants
  in any order. Let $S_1,S_2, S_3, S_4$ be similitudes that are
  defined by $S_i(x) := r(\kappa)x + z_i$, $i=1,2,3,4$.  By reasoning
  as in \cite[pp. 66--67]{Mat95} we have a non-empty compact set $K$
  in $Q$ such that
\begin{equation}\label{invariant}
K=S_1(K)\cup S_2(K)\cup S_3(K)\cup S_4(K).
\end{equation}
The similitudes $S_1, S_2, S_3 ,S_4$ satisfy an open set condition
\cite[p. 67]{Mat95}. Hence, by \cite[Corollary~5.8,
Theorem~4.14]{Mat95},
\[
\dim_\mathcal{M} (K)=\dim_\mathcal{H} (K)=-\frac{\log 4}{\log r(\kappa)}.
\]
If we let $\kappa$ vary between $(0,1/2]$, then $\dim_\mathcal{M} (K)$
reaches all values in $[1,2)$. There exists, in particular,
$\kappa=\kappa(\lambda)\in (0,1/2]$ for which the upper Minkowski
dimension of $K_\lambda:=K$ is $\lambda$.  We define $G$ to be the
open set
\[
G:=B^2(0,2)\setminus K_\lambda.
\]
Since $\partial G=\partial B^2(0,2)\cup K_\lambda$, we obtain
$\dim_\mathcal{M}(\partial G)=\lambda$. 

Estimate \eqref{paljon}
has been verified in \cite[Proposition 5.2]{HH-SV}. 

We shall estimate the John constant of the domain $G$.  We show that
there is a constant $c>0$ such that for every $x\in G$, there is a
rectifiable curve $\gamma:[0,\ell(\gamma)]\to G$, parametrized by its
arc length so that $\gamma(0)=x$, $\gamma(\ell(\gamma))=0$, and 
\begin{equation}\label{john_estimate}
\dist(\gamma(t),\partial G)\ge c\kappa t
\end{equation}
for all $t\in [0,\ell(\gamma)]$. Before the construction of $\gamma$,
let us explain why the property above implies that the John
constant $c_J$ of $G$ is larger than $\beta$ if $\lambda \le 2-\tau
\beta$, where $\tau := 4/(c\log 2)$.  By the mean value theorem 
(recall that $d \log_2(t)/dt = (t\log 2)^{-1}$ for $t>0$)
and the fact that $\kappa\in (0,1/2]$,
\begin{align*}
  \tau \beta \le 2-\lambda & = 2\bigg(1+\frac{1}{\log_2
    [(1-\kappa)/2]}\bigg) =
  2\bigg(\frac{\log_2(1-\kappa)}{\log_2(1-\kappa)-1}\bigg) \\
  & = 2\bigg(\frac{\log_2(1-\kappa)-\log_2(1)}{\log_2(1-\kappa)-\log_2(1)-1}\bigg) < \frac{2\kappa}{(\log 2)(1-\kappa)}\le \frac{4\kappa}{\log 2}.
\end{align*}
It follows that $\beta< c\kappa \le c_J$, where $c_J$ is the John
constant of $G$.

We construct the curve $\gamma$.  If $x$ lies in $G\setminus Q$ then
the construction is clear. Hence, we may assume without loss of
generality that $x$ lies in $G\cap Q=Q\setminus K_\lambda$.  Let $m\ge 0$ the
smallest positive integer for which $x\in Q\setminus K_\lambda^{m+1}$,
where for each $j\geq 1$
\begin{equation}\label{iter}
  K_\lambda^j:=\bigcup_{i_1=1}^4\dotsb \bigcup_{i_j=1}^4 S_{i_1}\circ\dotsb\circ S_{i_j}(Q).
\end{equation}
To see that such an $m$ exists, we use the fact that iterations
$K_\lambda^j$ converge to $K_\lambda$ in the Hausdorff metric $\rho$
as $j\to \infty$, \cite[pp. 66--67]{Mat95}.  Hence, there is $M\in \N$
such that $\rho(K_\lambda^M,K_\lambda)<\dist(x,K_\lambda)/2$.
Especially,
\[
K^M_\lambda \subset \{y\in \R^2\,:\,\dist(y,K_\lambda)\le \dist(x,K_\lambda)/2\},
\]
and it follows that $x\in Q\setminus K_\lambda^M$. Hence, the
smallest $m$ exists.

We write $X^0=\{y\in\R^2\,:\,y_1=0\text{ or }y_2=0\}$ for the
coordinate axes in $\R^2$, and
\[
X^j:=\bigcup_{i_1=1}^4\dotsb \bigcup_{i_j=1}^4 S_{i_1}\circ\dotsb\circ
S_{i_j}(X^0)
\]
for all $j\ge 1$. We connect $x$ to the set $X^{m}\cap Q$ by a curve
$\gamma:[0,t_m]\to G$ such that inequality \eqref{john_estimate} is
valid and
\[
t_m=\dist(x,X^{m})\le \kappa r(\kappa)^{m}< r(\kappa)^{m}.
\]
We proceed inductively: We connect sets $X^{j}$ and $X^{j-1}$ to each
other, when $j\ge 1$; we connect $X^{0}$ to $0$, when
$j=0$. Figure~\ref{fig:Fractal} depicts one of the intermediate
construction steps. Let us first consider the case $1\le j\le m$:
inequality \eqref{john_estimate} is valid for all $t\in [0,t_j]$,
point $\gamma(t_j)$ lies in $X^{j}\cap Q$, and
\begin{equation}\label{induction}
  t_j\le  \sum_{i=j}^\infty r(\kappa)^{i}.
\end{equation}
Let us connect $\gamma(t_j)$ to $X^{j-1}$ as follows: we define
$\gamma$ in $[t_j,t_{j-1}]$ by tracing along set $X^{j}\cap Q$ until
we reach $X^{j}\cap X^{j-1}\cap Q$.  This can be done in a way that
$t_{j-1}-t_j\le r(\kappa)^{j-1}$.  Hence,
\[
t_{j-1}= t_{j-1}-t_j + t_j \le r(\kappa)^{j-1} + \sum_{i=j}^\infty r(\kappa)^{i}\le \sum_{i=j-1}^\infty r(\kappa)^{i}\le 2 r(\kappa)^{j-1},
\]
and inequality \eqref{induction} is true for $j-1$.
Since $K_\lambda\subset K_\lambda^{m+1}$, we have
for all $t\in [t_j,t_{j-1}]$ 
\begin{align*}
  \dist(\gamma(t),\partial G) & =\dist(\gamma(t),K_\lambda) \ge
  \dist(\gamma(t),K_\lambda^{m+1}) \ge \kappa r(\kappa)^{j}
  \\ 
  & \qquad \ge 4^{-1}\kappa r(\kappa)^{j-1} \ge 8^{-1} \kappa t_{j-1}\ge
  8^{-1} \kappa t.
\end{align*}
Hence, inequality~\eqref{john_estimate} is valid for these values of
$t$.

We consider the case $j=0$. Now $\gamma(t_0)$
lies in $X^0\cap Q$, the curve $\gamma$  satisfies
inequality~\eqref{john_estimate} for all $t\in [0,t_0]$, and
$t_0\le \sum_{i=0}^\infty r(\kappa)^i$.
Define $\gamma$ in $[t_0,t_{-1}]$ 
by tracing along $X^0\cap Q$ during
time $t_{-1}-t_0\le 1$.
This yields a curve $\gamma:[0,t_{-1}]\to G$
satisfying
estimate~\eqref{john_estimate} for all $t\in [0,t_1]$.
\end{proof}

\begin{figure}[!htbp] 
\centering
\includegraphics[width=0.45\textwidth]{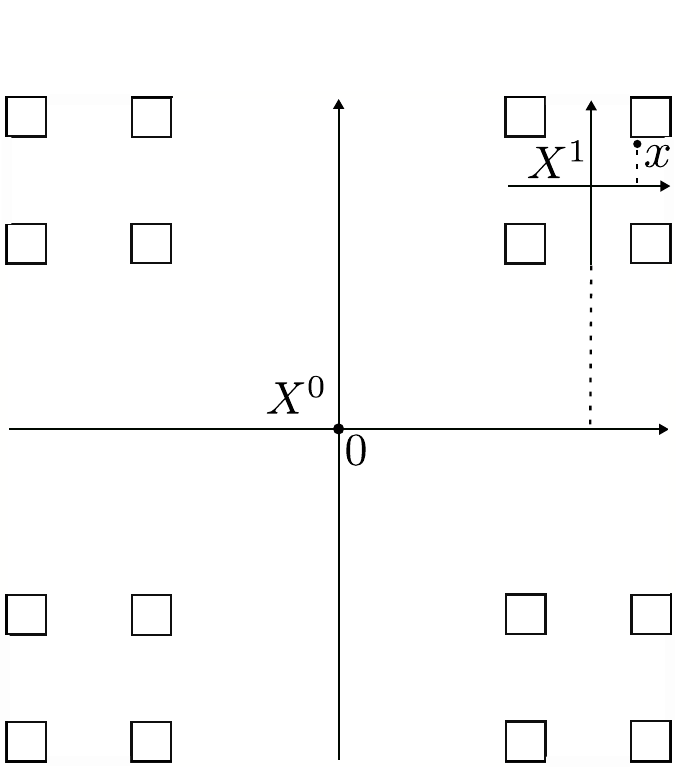}
\caption{$Q\setminus K_\lambda^2$ with coordinate axes $X^0$ and $X^1$
  ($\kappa=1/2$).}
\label{fig:Fractal}
\end{figure}

We are now ready for the proof of Theorem~\ref{sharp_counter_plane}.

\begin{proof}[Proof of Theorem~\ref{sharp_counter_plane}]
  Let $\varkappa\in (0,1)$ be as in Lemma~\ref{mod_lemma}
   and let $\tau>0$ be as in
  Proposition~\ref{olemassa}.  Choose ${\bar c_2}:=\max\{2/\varkappa,
  \tau/\varkappa\}$.  Fix $\lambda\in [1,2)$ and suppose that
  $\beta\in (0,1)$ is such that $\lambda \le 2-{\bar c_2}\, \beta$.
  Then 
  $\lambda\le 2-\tau \beta/\varkappa$, and hence by
  Proposition~\ref{olemassa} there exists a John domain $G$ in $\R^2$
  whose John constant is $c_J\ge \beta/\varkappa$, i.e.  $\beta \le
  \varkappa c_J$. By Theorem~\ref{thm:modification}, $G_\beta$ satisfies
  a $\beta/4$-quasihyperbolic boundary condition.
 Proposition~\ref{dimpres} gives that
$\dim_{\mathcal{M}}(\partial G_\beta)=\dim_{\mathcal{M}}(\partial G)
=\lambda.$
By Theorem~\ref{sharp_counter} the domain $G_\beta$ does not satisfy
the $(q,p)$-Poincar\'e inequality \eqref{poincare}.
\end{proof}

\begin{rem} \label{rmk:sharpness1}
  Theorem~\ref{sharp} states that a domain $G\subset\R^2$ supports the
  $(q,p)$-Poincar\'e inequality \eqref{poincare} with $1\le
  q<p<\infty$ if the following conditions are met: $G$ satisfies a
  $\beta$-quasihyperbolic boundary condition,
  $\mathrm{dim}_{\mathcal{M}}(\partial G)\le \lambda\in [1,2)$, and
  \begin{equation}\label{eq:threshold}
    p>\frac{q(2-\lambda b)}{q+b(2-\lambda)},\qquad b=\frac{2\beta}{1+\beta}.
\end{equation}
Recall a natural restriction among these parameters; by
Theorem~\ref{KR_theorem} there is a constant $c_2>0$ such that
$\mathrm{dim}_{\mathcal{M}}(\partial G)\le 2-c_2\, \beta$.  Taking
this into account, Theorem~\ref{sharp_counter_plane} shows that our
main result is essentially sharp in essentially all the possible cases
when $n=2$. More precisely, if $\lambda \le 2-\bar c_2\,\beta$, then
there is a domain $G_\beta$ which satisfies a
$\beta/4$-quasihyperbolic boundary condition,
$\mathrm{dim}_{\mathcal{M}}(\partial G_\beta)=\lambda$, and $G_\beta$
does not support the $(q,p)$-Poincar\'e inequality~\eqref{poincare} if
the inequality in \eqref{eq:threshold} fails to hold and $1\le
q<p<\infty$.
\end{rem}

\section{Applications to the Neumann problem}
\label{sect:appls}

We discuss an application of our main result to the study of the
solvability 
of the elliptic second order Neumann problem in quasihyperbolic
boundary condition domains. We hence supplement Corollary~4.5 in
\cite{KOT}.

Let $G$ in $\R^n$ be a bounded domain which satisfies a
$\beta$-quasihyperbolic boundary condition. We consider the following
second order strongly elliptic partial differential operator
\begin{equation} \label{operator} \mathcal{L}_Au =
  -\sum_{i,j=1}^n\frac{\partial}{\partial
    x_i}\left(a_{ij}(x)\frac{\partial u}{\partial x_j}\right),
\end{equation}
where $A=\{a_{ij}(x)\}_{i,j=1}^n$, with $a_{ij}$ being real-valued
measurable functions in $G$, $a_{ij}=a_{ji}$, and there exists a
constant $c\geq 1$ such that
\[
c^{-1}|\xi|^2\leq \sum_{i,j=1}^na_{ij}(x)\xi_i\xi_j\leq
c|\xi|^2
\]
for almost every $x\in G$ and for all $\xi\in\R^n$. We deal with the
equation in divergence form
\begin{equation} \label{eq}
\mathcal{L}_Au = f
\end{equation}
subject to the Neumann boundary condition.

Let us fix $1\leq q<\infty$. We say that $u$ is in the domain
$\mathcal{D}(\mathcal{L}_A)$ of $\mathcal{L}_A$ if $u$ satisfies
\eqref{eq} in the weak sense, i.e. $u\in W^{1,2}(G)\cap L^q(G)$,
$f\in L^{q'}(G)$ with $q'=q/(q-1)$, and
\[
\int_G\sum_{i,j=1}^na_{ij}(x)\frac{\partial u}{\partial x_i}\frac{\partial
  \phi}{\partial x_j}\, dx = \int_G \phi(x)f(x)\, dx
\]
for all $\phi\in W^{1,2}(G)\cap L^q(G)$. Then the Neumann problem is
said to be $q$-solvable if for each $f\in L^{q'}(G)$ with 
\[
\int_G f\, dx = 0
\]
there exists $u\in \mathcal{D}(\mathcal{L}_A)$ such that \eqref{eq}
holds in the weak sense.

We have the following corollary of Theorem~\ref{sharp}.

\begin{cor}
  Suppose $G\subset\R^n$, $n\geq 2$, satisfies a
  $\beta$-quasihyperbolic boundary condition for some $\beta \in (0,1]$ such that 
  $(n-2)/(n+2)\leq \beta$. Then the Neumann problem \eqref{eq}
  on $G$ is $q$-solvable for each $1\leq q<2$.
\end{cor}

\begin{proof}
  By \cite[6.10.1/Lemma, p. 381]{M}, the Neumann problem \eqref{eq} on
  $G$ is $q$-solvable if and only if $G$ supports a $(q,2)$-Poincar\'e
  inequality. Therefore, the claim follows from Corollary~\ref{k}.
\end{proof}

We note that in the plane the Neumann problem \eqref{eq} is
$q$-solvable in every $\beta$-quasihyperbolic boundary condition
domain $G$, $0<\beta\le 1$, whenever $1\leq q<2$.

\end{document}